\documentclass[11pt]{amsart}

\usepackage{amscd,amssymb}
\usepackage[all]{xy}
\usepackage[plainpages,backref,urlcolor=blue]{hyperref}
\usepackage[utf8]{inputenc}
\usepackage{mathtools}

\theoremstyle{plain}
\newtheorem{thm}[subsection]{Theorem}
\newtheorem{lem}[subsection]{Lemma}
\newtheorem{prop}[subsection]{Proposition}
\newtheorem{cor}[subsection]{Corollary}
\usepackage{xcolor}
\theoremstyle{definition}
\newtheorem{rk}[subsection]{Remark}
\newtheorem{definition}[subsection]{Definition}
\newtheorem{ex}[subsection]{Example}

\numberwithin{equation}{section}
\newcommand{\OO}{\mathcal O}

\newcommand{\A}{\mathcal A}
\newcommand{\B}{\mathcal B}

\newcommand{\CC}{\mathcal C}

\newcommand{\V}{\mathcal V}

\newcommand{\al}{\alpha}

\newcommand{\HH}{\mathcal H}

\newcommand{\Z}{\mathbb Z}
\newcommand{\Q}{\mathbb Q}
\newcommand{\R}{\mathbb R}
\newcommand{\C}{\mathbb C}
\newcommand{\PP}{\mathbb P}
\newcommand{\N}{\mathbb N}

\DeclareMathOperator{\rank}{rank}

\begin{document}

\title[On the Jacobian algebras of Ziegler pairs of plane arrangements]{On the Jacobian algebras of Ziegler pairs of plane arrangements}

\author[Alexandru Dimca]{Alexandru Dimca$^1$}
\address{Universit\'e C\^ote d'Azur, CNRS, LJAD, France and Simion Stoilow Institute of Mathematics,
P.O. Box 1-764, RO-014700 Bucharest, Romania}
\email{Alexandru.Dimca@univ-cotedazur.fr}

\author[Piotr Pokora]{Piotr Pokora$^{2}$}
\address{Department of Mathematics,
University of the National Education Commission Krakow,
Podchor\c a\.zych 2,
PL-30-084 Krak\'ow, Poland.}
\email{piotrpkr@gmail.com, piotr.pokora@uken.krakow.pl}

\thanks{$^1$ Partial support from the project ``Singularities and Applications'' - CF 132/31.07.2023 funded by the European Union - NextGenerationEU - through Romania's National Recovery and Resilience Plan. $^2$ Supported by the National Science Centre (Poland) Sonata Bis Grant \textbf{2023/50/E/ST1/00025}. For the purpose of Open Access, the author has applied a CC-BY public copyright license to any Author Accepted Manuscript (AAM) version arising from this submission.}



\begin{abstract}
We consider a Ziegler pair of plane arrangements, that is two plane arrangements $\A:f=0$ and $\A':f'=0$ in the projective space $\PP^3$, such that the intersection lattices $L(\A)$ and $L(\A')$ are isomorphic, but the Betti numbers of the minimal resolutions of their Jacobian algebras are not the same. We introduce several properties for such pairs and relate them to cones over Ziegler pairs of line arrangements in $\PP^2$.

\end{abstract}

\maketitle

\section{Introduction}

Let $S=\C[x,y,z,w]$ be the polynomial ring in four variables $x,y,z,w$ with complex coefficients, and let $X:f=0$ be a reduced surface of degree $d\geq 3$ in the complex projective space $\PP^3$. 
We denote by $J_f$ the Jacobian ideal of $f$, i.e. the homogeneous ideal in $S$ spanned by the partial derivatives $f_x,f_y,f_z,f_w$ of $f$, and  by $M(f)=S/J_f$ the corresponding graded quotient ring, called the Jacobian (or Milnor) algebra of $f$.
Consider the general form of the minimal resolution of the Milnor algebra $M(f)$ of a reduced surface $X:f=0$, namely
\begin{equation}
\label{res2A}
0 \longrightarrow \bigoplus_{k=1}^{r} S(1-d-b_k)
 \longrightarrow \bigoplus_{j=1}^{q} S(1-d-c_j)
   \longrightarrow \bigoplus_{i=1}^{p} S(1-d-d_i)
   \longrightarrow S^4(1-d)
   \longrightarrow S,
\end{equation}
where $p \geq 3$, $q \geq 0$ and $r \geq 0$. Equivalently, the minimal resolution of the graded $S$-module $D_0(f)$ of derivations killing the polynomial $f$, that is
$$D_0(f)= \{\delta =\al^x\partial_x + \al^y \partial_y+ \al^z \partial_z+\al^w \partial_w \ | \ \delta(f)=0\}$$
with $\al^x, \al^y, \al^z, \al^w \in S$ is given by
\begin{equation}
\label{resD0}
0 \longrightarrow \bigoplus_{k=1}^{r} S(-b_k)
 \longrightarrow \bigoplus_{j=1}^{q} S(-c_j)
   \longrightarrow \bigoplus_{i=1}^{p} S(-d_i)
   \longrightarrow D_0(f).
   \end{equation}
We call the increasing sequences of degrees 
$${\bf d}=(d_1, \ldots, d_p), \  
{\bf c}=(c_1, \ldots, c_q) 
 \text{ and } {\bf b}=(b_1, \ldots, b_r)$$ 
  the {\it graded Betti numbers of the Jacobian algebra} $M(f)$, since they determine and are determined by the usual
graded Betti numbers of the Jacobian algebra $M(f)$ as defined for instance in \cite{Eis}. 

It is known that there is a unique polynomial $P(M(f))(u) \in \Q[u]$, called the {\it Hilbert polynomial} of $M(f)$, and an integer $k_0\in \N$ such that
\begin{equation}
\label{Hpoly}
 H(M(f))(k)= P(M(f))(k)
\end{equation}
for all $k \geq k_0$. We denote by $\Sigma$ the singular subscheme of $X$, which is defined by the Jacobian ideal $J_f$. The general theory of Hilbert polynomials says that the degree of  $P(M(f))$ is given by the dimension of the support of  $\OO_{\Sigma}$, the coherent sheaf associated to the graded $S$-module $ M(f)$. Hence the assumption $\dim \Sigma=0$ implies that the polynomial
$P(M(f))$ is a constant, namely the total Tjurina number of $X$, given by
\begin{equation}
\label{ab0}
P(M(f))=\tau(X)=\sum_{s \in \Sigma} \tau(X,s),
\end{equation}
where $\tau(X,s)$ denotes the Tjurina number of the isolated singularity $(X,s)$, and
$\dim \Sigma=1$ implies that 
\begin{equation}
\label{ab1}
P(M(f))(u)=au+b,
\end{equation}
where $a=\deg(\Sigma)$, the degree of the subscheme $\Sigma$. 
There is a lot of interest in understanding the relations between these numerical invariants coming from the resolution of $M(f)$ and the combinatorics of $X$, when $X$ is a plane arrangement in $\PP^3$.
In fact, we introduce in this paper the following properties, which may be stated for any pair of hyperplane arrangements.
\begin{definition}
\label{defZP}
Two hyperplane arrangements $\A:f=0$ and $\A':f'=0$ in $\PP^n$ are said to form a Ziegler pair if they have isomorphic intersection lattices $L(\A)$ and $L(\A')$, obtained from the corresponding central arrangements in $\C^{n+1}$, and distinct sets of graded Betti numbers coming from the minimal resolutions of their Jacobian algebras $M(f)$ and $M(f')$. 

A Ziegler pair  $\A:f=0$ and $\A':f'=0$ may have one or several of the following additional properties.
\begin{enumerate}
\item We say that the Ziegler pair  $\A:f=0$ and $\A':f'=0$  satisfies the condition ${\rm (HP)}$ if the Jacobian algebras $M(f)$ and $M(f')$ have the same Hilbert polynomial,  that is 
$$\dim M(f)_k =\dim M(f')_k \text{ for any integer } k \geq k_0$$
for some $k_0 \in \Z$.
\item We say that the Ziegler pair  $\A:f=0$ and $\A':f'=0$  satisfies the condition ${\rm (HF)}$ if the Jacobian algebras $M(f)$ and $M(f')$ have the same Hilbert function, that is 
$$\dim M(f)_k =\dim M(f')_k \text{ for any integer } k.$$
\item We say that the Ziegler pair $\A:f=0$ and $\A':f'=0$ satisfies condition ${\rm (MDR)}$ if the Jacobian algebras $M(f)$ and $M(f')$ have the same minimal degree of a first Jacobian syzygy. More precisely, using the notation from \eqref{resD0}, we set
\[{\rm mdr}(f)=d_1=\min_i \{d_i\}.\]
Then the condition ${\rm (MDR)}$ is equivalent to
\[ {\rm mdr}(f) = {\rm mdr}(f'). \]
\item  We say that the Ziegler pair $\A:f=0$ and $\A':f'=0$ satisfies condition $({\rm SPEC}_0)$ if the arrangement $\A$ is a specialization of the arrangement $\A'$. More precisely, this means that there exists an interval $I=[0,t_0]\subset \R$ and a smooth family of homogeneous polynomials $f_t$, for $t\in I$, such that
\[f_0=f, \qquad f_{t_0}=f',\]
and the corresponding family of hyperplane arrangements $\A_t:f_t=0$ has constant intersection lattice, i.e.
\[L(\A_t) \cong L(\A_{t'}) \quad \text{for all } t,t'\in I.\]
If, in addition, the graded Betti numbers of the minimal resolutions of the Jacobian algebras $M(f_t)$ are constant for all $0<t\leq t_0$, then we say that the Ziegler pair $\A:f=0$ and $\A':f'=0$ satisfies condition $({\rm SPEC})$.
\end{enumerate}

\end{definition}
Examples of Ziegler pairs satisfying or not satisfying some of these conditions will be given below. Note that for line arrangements in 
$\PP^2$, a Ziegler pair always satisfies the condition ${\rm (HP)}$, since in this case we have $P(M(f))=\tau(\A)$ and $\tau(\A)$ is determined by the intersection lattice $L(\A)$.

It came as a surprise that for a plane arrangement $\A:f=0$ in $\PP^3$
the Hilbert polynomial is not determined by the intersection lattice $L(\A)$ of $\A$. Indeed, the Ziegler pair of plane arrangements
$$\A:f=xyzw(x + y + z)(2x + y + z)(2x + 3y + z)(2x + 3y + 4z)(3x + 7z)(3x + 5y + 7z)=0$$
and
$$\A':f'=xyzw(x + y + z)(2x + y + z)(2x + 3y + z)(2x + 3y + 4z)(x + 3z)(x + 2y + 3z)=0$$
have 
$$P(M(f))(u)=51u-223 \text{ and } P(M(f'))(u)=51u-222.$$
This is a slight modification of \cite[Example 4.6]{MNS}, where the last two factors of $f$ are $(3x+5z)(3x+4y+5z)$; the Hilbert polynomials are the same. We use this modification since $\A$ and $\A'$ are then connected by a real family of arrangements with constant intersection lattice, see Example \ref{exI1}.
We introduce the following definition, which is widely used in hyperplane arrangement theory, see \cite{DHA,OT}. 
\begin{definition}
\label{defCone}
Let $R=\C[x,y,z]$ and consider a reduced curve $C:g=0$ in $\PP^2$, where $g \in R$ is a homogeneous polynomial of degree $d-1$.
The surface $S=cC:f=0$, where $f=wg$, is called the cone over the curve $C$.
\end{definition}
This terminology, which seems strange at first, is perhaps motivated by the fact that the complement
$U(f)= \PP^3 \setminus cC$ is exactly the complement in $\C^3$ of the usual cone over $C$ given in $\C^3$ by the equation $g=0$.

Note that the above plane arrangements are in fact cones over a Ziegler pair of line arrangements, namely $\A=c\B$ and $\A'=c\B'$, where
\begin{equation}
\label{exB}
\B:g=xyz(x + y + z)(2x + y + z)(2x + 3y + z)(2x + 3y + 4z)(3x + 7z)(3x + 5y + 7z)=0
\end{equation}
and
\begin{equation}
\label{exB'}
\B':g'=xyz(x + y + z)(2x + y + z)(2x + 3y + z)(2x + 3y + 4z)(x + 3z)(x + 2y + 3z)=0. 
\end{equation}
More information on this Ziegler pair is given in Example \ref{exI1}, where we show that it satisfies the condition ${\rm (SPEC)}$ and does not satisfy the condition $({\rm HF})$.
The first  result of this note says that we can use any such Ziegler pair of line arrangements to produce similar examples.

 \begin{thm}
\label{thm1}
Let $\B$ and $\B'$ be a pair of line arrangements in $\PP^2$.
\begin{enumerate}
\item If $\B:g=0$ and $\B':g'=0$ have isomorphic intersection lattices, then the cones $\A=c\B: f=wg=0$ and $\A'=c\B':f'=wg'=0$ have also  isomorphic intersection lattices.
\item If $\B:g=0$ and $\B':g'=0$ is a Ziegler pair satisfying the condition ${\rm (SPEC)}$ and not satisfying the condition $({\rm HF})$, then
$$P(M(f)) \ne P(M(f')).$$
\end{enumerate}

\end{thm}

\begin{rk}
\label{rkC}
It is clear that the claim in Theorem \ref{thm1} (1) is valid, with essentially the same proof, for a pair $(\B,\B')$ of hyperplane arrangements in $\PP^n$. In particular, if  $(\B,\B')$ is a pair of line arrangements in $\PP^2$ with isomorphic intersection lattices, and we take the $k$-fold cones, that is if we set
$\A_k: w_1 \cdots w_k g=0$ and $\A_k': w_1 \cdots w_k g'=0$, where $w_1, \ldots , w_k$ are new variables, then the hyperplane arrangements $\A_k$ and $\A_k'$ in $\PP^{k+2}$ have isomorphic intersection lattices.
\end{rk}
The following result shows the big difference with the curve case, where
the Hilbert polynomial $P(M(g))$ is determined by the total Tjurina number of $C$, which for a line arrangement $\B:g=0$ is determined by the topology of the complement
$U(g)=\PP^2 \setminus \B$.
\begin{cor}
\label{corC} 
For plane arrangements $\A:f=0$ in $\PP^3$, the Hilbert polynomial $P(M(f))$ is determined neither by the topology of the complement $U(f)= \PP^3 \setminus \A$, nor by the topology of the corresponding Milnor fibration
$$\C^4 \setminus \{(x,y,z,w) \in \C^4 \ : \ f(x,y,z,w)=0\} \to \C^*, \ (x,y,z,w) \mapsto f(x,y,z,w).$$ 

\end{cor}

The second main result shows that the surfaces which are cones over curves enjoy nice properties in relation with tameness, a subtle property introduced and studied in \cite{ DmaxS,HD}.

\begin{thm}
\label{thm2}
Let $C:g=0$ be a reduced curve in $\PP^2$ of degree $d-1$ with ${\rm mdr}(g) > 0$. Consider the cone $S=cC:f=wg=0$ in $\PP^3$.
\begin{enumerate}

\item If $r_j=(a_j,b_j,c_j)$ for $j=1,\ldots,m$ is a minimal set of generators for the first syzygy $R$-module $D_0(g)$, then
$$\rho_0=(x,y,z,-(d-1)w)$$
and
$$\rho_j=(a_j,b_j,c_j,0)$$
 for $j=1,\ldots,m$ is a minimal set of generators for the first syzygy $S$-module $D_0(f)$. In particular, one has ${\rm mdr}(f) =1$. Moreover, if the minimal resolution for $D_0(g)$ is given by \eqref{res2A2},
then the minimal resolution for $D_0(f)$ has the form
\begin{equation}
\label{res2D01}
0 
 \longrightarrow \bigoplus_{j=1}^{q} S(-c_j)
   \longrightarrow S(-1) \oplus \left(\bigoplus_{i=1}^{p} S(-d_i)\right)
   \longrightarrow D_0(f).
\end{equation}
\item The surface $S=cC:f=wg=0$ is tame with respect to the pair $(\rho_0,\rho_1)$.
\item The surface $S$ is free (resp. nearly free, resp. strictly plus-one generated with respect to $\overline \rho_S=(\rho_0,\rho_1,\rho_2, \rho_3)$) if and only if the curve $C$ is free (resp. nearly free, resp. plus-one generated  with respect to $\overline \rho_C=(\rho_1,\rho_2, \rho_3)$).

\end{enumerate}

\end{thm}

\begin{rk}
\label{rkS}
It is clear that the claim in Theorem \ref{thm2} (1) is valid, with essentially the same proof, for a  hypersurface in $\PP^n$. In particular, if  $C:g=0$ is a curve in $\PP^2$, and we take the $k$-fold cones, that is if we set
$D: f_k=w_1 \cdots w_k g=0$, where $w_1, \ldots , w_k$ are new variables, then $D_0(f_k)$ is generated by $\rho_j$ for $j=1,\ldots,m$ and $k$ additional syzygies, namely
$$\theta_i=x\partial_x+y\partial_y+z\partial_z - dw_i\partial_{w_i}$$
for $i=1,\ldots,k$, where $\deg g = d$.
\end{rk}

One may ask if plane arrangements which are not cones over line arrangements behave nicer with respect to Hilbert polynomials. Example \ref{ex1} in Section 4 shows that this is not the case. To better understand this example, we study a related moduli space in Proposition \ref{ellp}. 

Finally in Section 5 we construct examples of Ziegler pairs of line arrangements in $\PP^2$ which satisfy the condition $({\rm HF})$ in Definition \ref{defZP}.

\section{Proof of the main results}

\subsection{Proof of Theorem \ref{thm1}}
To prove (1), note that a family of lines $L_1, \ldots, L_s$ in $\B$ intersect in a point in $\PP^2$ if and only if the corresponding planes in $\A$,
that is those defined by the same equations, intersect in a line $L$ in $\PP^3$. Since any line $L$ in $\PP^3$ intersects the plane $P:w=0$, it follows that the lattice $L(\B)$ determines the lattice $L(\A)$.

To prove the claim (2), we use the following equality, see \cite[Proposition 3.13]{FS},
$$P(M(f))(u)=au+b,$$
where
$$a=\tau(\B)+d-1$$
where $d-1$ is the number of lines in $\B$, and
$$b=\sum_{j=0}^{st(g)-1} H(M(g))(j)-st(g)\tau(\B)-\frac{(d-4)(d-1)}{2}.$$
Here $H(M(g))(j) =\dim M(g)_j$ is the Hilbert function of the Jacobian algebra $M(g)$ and
$$st(g)=\min\{j_0 \ : \ H(M(g))(j)=\tau(\B) \text{ for any } j \geq j_0\}$$ 
is the stabilization threshold of $g$. It is clear from the above formulas that in the equality for $b$ we may replace $st(g)$ by any larger integer, namely by $T=3(d-2)+1$.
Since $(\B,\B')$ satisfies the condition ${\rm (SPEC)}$, we get the inequality by the semicontinuity of the Hilbert functions
$$\sum_{j=0}^{T-1} H(M(g))(j) \geq \sum_{j=0}^{T-1} H(M(g'))(j).$$
Since $(\B,\B')$ does not satisfy the condition $({\rm HF})$, it follows that this inequality is strict and hence
$$P(M(f))(0)=b \neq b'=P(M(f'))(0).$$
\subsection{Proof of Theorem \ref{thm2}}
Let $\rho'=a' \partial_x+b' \partial_y+c '\partial_z+d' \partial_w$ be a derivation in $D_0(f)$. Then we get
$$w(a'g_x+b'g_y+c'g_z)+d'g=0.$$
If $d'=0$, then $\rho'$ is in the $S$-submodule $D'$ generated in $D_0(f)$ by the derivations $\rho_j$'s for $j=1,\ldots,m$. If $d' \ne 0$, then $d'=wd_1$ for some $d_1 \in S$, and hence
$$\rho'+\frac{d_1}{d-1}\rho_0 \in D'.$$
Hence $\rho_j$ for $j=0,\ldots,m$ minimally generate $D_0(f)$, and clearly all secondary syzygies involving the $\rho_j$ come from the secondary syzygies of $g$. This proves the first claim (1).

To prove (2), consider the $2 \times 4$ matrix $M(\rho_0, \rho_1)$ having as first row the components of $\rho_0$ and as second row the components of $\rho_1$. Let $I_2$ be the ideal in $S$ generated by the 2-minors of $M(\rho_0, \rho_1)$. Note that
$wa_1,wb_1$ and $wc_1$ are in $I_2$. If $w \ne 0$, then the equations
$$a_1=b_1=c_1=0$$
define a one-dimensional subspace in $\PP^3$.
Indeed, $a_1=b_1=c_1=0$ define a 0-dimensional subspace in $\PP^2$,
since $a_1,b_1,c_1$ have no common factor in $R$, the syzygy $\rho_1$ being primitive.
If $w=0$, then note that the matrix $M'=M(\rho_0, \rho_1)(x,y,z,0)$ 
has rank 2 over the fraction field of $R$, since a relation
$$aE+b\rho_1=0$$
with $a,b \in R$ not both 0 and $E=(x,y,z)$ is impossible.
It follows that the 2-minors in $M'$ define subset in $\PP^2$ of dimension at most 1. These two cases show that the ideal $I_2$ defines a subset of codimension at least 2 in $\C^4$, and the claim (2) follows using \cite[Lemma 2.3 (ii)]{DmaxS} or \cite[Theorem 3.3 and Remark 3.4]{HD}.

To prove the claim in (3) about the free surface $S$, we apply Saito's Freeness Criterion, see \cite[Theorem 8.1]{DHA}. This claim is then clear, since the corresponding two determinants $\phi(g)$ and $\phi(f)$
are related by the formula
$$\phi(f)=-d\,w\,\phi(g)$$
when we use generators for $D_0(f)$ and $D_0(g)$ as in the claim (1).
An alternative proof can be obtained using \cite[Theorem 4.1 (4)]{HD}.
The claims about nearly free or strictly plus-one generated surfaces with respect to
$$\overline \rho=(\rho_0,\rho_1,\rho_2, \rho_3)$$
 follow from
\cite[Definition 1.1]{HD} and the fact that the second order syzygies of $g$ and $f$ coincide. For the definition of (strictly) plus-one generated hyperplane arrangements we refer to \cite{Abe}, where this notion was introduced. We recall that a plane curve is strictly plus-one generated if and only if it is  plus-one generated, see \cite{Abe,3syz}.

\section{Some examples of cones and the proof of Corollary \ref{corC}}
\begin{ex}
\label{exI1}
It is straightforward to verify that the Ziegler pair $(\B,\B')$ from \eqref{exB} and \eqref{exB'} satisfies neither condition $({\rm HF})$ nor condition ${\rm (MDR)}$. Indeed, if the minimal graded resolution of the $R$-module $D_0(g)$ is given by
 \begin{equation}
\label{res2A2}
0 \longrightarrow \bigoplus_{j=1}^{q} R(-c_j)
   \longrightarrow \bigoplus_{i=1}^{p} R(-d_i)
   \longrightarrow D_0(g),
\end{equation}
then we call the increasing sequences of degrees 
$${\bf d}=(d_1, \ldots, d_p) \text{ and }  
{\bf c}=(c_1, \ldots, c_q) $$ 
  the {\it graded Betti numbers of the Jacobian algebra} $M(g)$. Note that there is a small shift in notation when compared to \cite{Betti2}.
Then, a direct computation shows that the graded Betti numbers of $M(g)$ are
$${\bf d}=(6_6) \text{ and } {\bf c}=(7_4)$$
with the obvious notations as explained in \cite{Betti3},
while the graded Betti numbers of $M(g')$ are
$${\bf d}=(5,6_3) \text{ and } {\bf c}=(7,8).$$
Geometrically, $\B$ and $\B'$ are constructed from a hexagon $\HH$, by adding the 3 main diagonals. The 6 vertices of the hexagon are situated on a conic for $\B'$, but not for $\B$, see \cite{ZP} for more details on this construction. This 
Ziegler pair $(\B',\B)$ satisfies the condition ${\rm (SPEC)}$. Indeed, consider the family
$$f_t=xyz(x + y + z)(2x + y + z)(2x + 3y + z)(2x + 3y + 4z)$$
$$((1+2t)x + (3+4t)z)((1+2t)x + (2+3t)y + (3+4t)z)$$
for $t \in [0,1]$. Then $f_0=g'$, $f_1=g$, the intersection lattice of $f_t=0$ is constant for $t\geq 0$, and the six triple points lie on a conic only for $t=0$. For $t>0$, \cite[Theorem 1.8]{ZP} gives ${\rm mdr}(f_t)=6$. Since $\tau(f_t=0)=42$ is the du Plessis--Wall upper bound for curves of degree $9$ with ${\rm mdr}=6$, the curves $f_t=0$ are maximal Tjurina curves, and hence their graded Betti numbers are ${\bf d}=(6_6)$, ${\bf c}=(7_4)$ for all $t>0$. Therefore the family $f_t$ satisfies all the properties from Definition \ref{defZP} (4).

By writing the Hilbert function as a Hilbert series, one has the following equalities
$$HF(g)=1+3t+6t^2+10t^3+15t^4+21t^5+28t^6+36t^7+42t^8+46t^9+48t^{10}+$$
$$+48t^{11}+46t^{12}+42t^{13}+\ldots$$
and
$$
HF(g')=1+3t+6t^2+10t^3+15t^4+21t^5+28t^6+36t^7+42t^8+46t^9+48t^{10}+$$
$$+48t^{11}+46t^{12}+43t^{13}+42t^{14}+ \ldots$$
If we write $P(M(f))(u)=au+b$ and $P(M(f'))(u)=a'u+b'$, we have
$d=10$, $st(g)=13$ and $st(g')=14$. The above formulas yield
$$a=a'= 42+(10-1)=51$$
and
$$b=-223  \text{ and } b'=-222.$$

\end{ex}
\subsection{Proof of Corollary  \ref{corC}}

To prove this claim, it is enough to recall that two central hyperplane arrangements that are lattice-isotopic have equivalent Milnor fibrations and complements, see \cite[Theorem 5.1]{DHA} for the precise statement and \cite{R} for the proof of this deep result. Our discussion above in Example \ref{exI1} shows that $f$ and $f'$ define lattice-isotopic central arrangements in $\C^4$, and hence all the topological  invariants of the corresponding Milnor fibrations or of the complements $U(f)$ and $U(f')$ are the same.
\begin{ex}
\label{ex0}
Using the graded Betti numbers of the Ziegler pair of line arrangements
$\B:g=0$ and $\B':g'=0$ from Example \ref{exI1} and Theorem \ref{thm2} (1), we get the graded Betti numbers of $M(f)$ where $f=wg$ as follows
$${\bf d}=(1,6_6) \text{ and } {\bf c}=(7_4),$$
while the graded Betti numbers of $M(f')$, where $f'=wg'$, are
$${\bf d}=(1,5,6_3) \text{ and } {\bf c}=(7,8).$$

Let us consider the following hyperplane arrangements in $\mathbb{P}^4$:
{\small
\[
\mathcal{D} : h = xyzw_{1}w_{2}(x + y + z)(2x + y + z)(2x + 3y + z)(2x + 3y + 4z)(3x + 7z)(3x + 5y + 7z) = 0,
\]
and
\[
\mathcal{D}' : h' = xyzw_{1}w_{2}(x + y + z)(2x + y + z)(2x + 3y + z)(2x + 3y + 4z)(x + 3z)(x + 2y + 3z) = 0,
\]}
which are double cones over the line arrangements in Example \ref{exI1}.
In light of Remark~\ref{rkS}, we can anticipate the structure of the modules $D_{0}(h)$ and $D_{0}(h')$. In both cases, besides the derivations inherited from the corresponding arrangements in the $(x,y,z)$-variables, there are two additional derivations given by
\[
\theta_i = x\partial_x + y\partial_y + z\partial_z - 9\,w_i\partial_{w_i}, \qquad i \in \{1,2\}.
\]
Thus, the original derivations persist, and the extension to $\mathbb{P}^4$ contributes these two extra logarithmic derivations corresponding to the variables $w_1$ and $w_2$.
\end{ex}

\section{Some hyperplane sections of double cones}

We start with an example of a Ziegler pair not involving cones and still failing property $({\rm HP})$.
\begin{ex}
\label{ex1}
For
$$\A: f= (w+x-3y+5z)wxy(x-y-z)(x-y+z)(2x+y-2z)(x+3y-3z)$$
$$(3x+2y+3z)(x+5y+5z)(7x-4y-z)=0$$
we get that the Hilbert polynomial is $61u-307$,
while for 
$$\A': f'=(w+x-3y+5z)wxy(4x-5y-5z)(x-y+z)(16x+13y-20z)$$
$$(x+3y-3z)(3x+2y+3z)(x+5y+5z)(7x-4y-z)=0$$
we get the Hilbert polynomial $61u-308$.
It is not difficult to show that $\A$ and $\A'$ are irreducible, hence not cones, and have isomorphic intersection lattices. In fact, the graded Betti numbers of $M(f)$ as defined above are
$${\bf d}=(2,6_2,7_7), \ {\bf c}=(7,8_7,9_2) \text{ and } {\bf b}=(9_2,10),$$
 while the 
graded Betti numbers of $M(f')$  are
$${\bf d}=(2,7_{13}), \ {\bf c}=(8_{16}) \text{ and } {\bf b}=(9_5).$$
Hence ${\rm mdr}(f)={\rm mdr}(f') >1$, which shows that $\A$ and $\A'$ are not cones using Theorem \ref{thm2} (1).

Geometrically, this example is obtained as follows. We start with a pair of Ziegler line arrangements and apply twice the coning construction. In this way we obtain a pair of hyperplane arrangements $\CC$ and $\CC'$ in $\PP^4$, having isomorphic intersection lattices, as explained in Remark \ref{rkC}. Then $\A$ and $\A'$ are obtained as the hyperplane section $w_2=w_1+L$, $L=x-3y+5z$, of $\CC$ and $\CC'$. This section is not generic: the line $L=0$ passes through the double point $y=x+5y+5z=0$ of both line arrangements, which creates the extra quadruple point $[x:y:z:w]=[5:0:-1:0]$. One checks directly that $L(\A)\cong L(\A')$.
\end{ex}
\begin{rk}
\label{rkGen}
Hyperplane sections of the same double cones in general position also give Ziegler pairs not satisfying ${\rm (HP)}$. Let $g$ and $g'$ be the products of the nine linear factors of $f$ and $f'$ in Example \ref{ex1} which do not involve $w$, see also Remark \ref{rkXcomp}. For instance, for $L=2x-5y+11z$, the line $L=0$ passes through no intersection point of the line arrangements $g=0$ and $g'=0$, and the arrangements
$$\A_1: (w+L)wg=0 \quad \text{and} \quad \A_1': (w+L)wg'=0$$
have isomorphic intersection lattices. Their Hilbert polynomials are $61u-308$ and $61u-309$, respectively, and the graded Betti numbers of the corresponding Jacobian algebras are
$${\bf d}=(2,6_2,7_6,8), \ {\bf c}=(7,8_5,9_4) \text{ and } {\bf b}=(9,10_2)$$
for $\A_1$, and
$${\bf d}=(2,7_{12},8), \ {\bf c}=(8_{14},9_2) \text{ and } {\bf b}=(9_4,10)$$
for $\A_1'$.
\end{rk}
Since the Ziegler pair $(\mathcal{A}, \mathcal{A}')$ in Example \ref{ex1} has an interesting geometry, we now determine the moduli space of arrangements that share the same intersection lattice $L(\mathcal{A})$.

In order to do so we need a short preparation.
Let $M = (E, \mathcal{B}(M))$ be a rank $n$ matroid on a ground set $E$ with $|E| = d$, realizable over a field $\mathbb{F}$. Recall that a realization of $M$ is an arrangement of hyperplanes
\[
\mathcal{A} = \{H_1,\dots,H_d\} \subset \mathbb{P}^{n-1}_{\mathbb{F}}
\]
such that
\[
\mathcal{B}(M) = \left\{ \lambda \subseteq E : |\lambda| = n,\ \bigcap_{i\in\lambda} H_i = \emptyset \right\}.
\]

Choosing homogeneous coordinates, we represent $\mathcal{A}$ by a family $\V$ of vectors $v_i \in  \mathbb{P}^{n-1}_{\mathbb{F}}$ for $i \in E$, where $v_i$ defines the hyperplane $H_i$. Then $\V$ realizes $M$ if and only if for every $\lambda \subseteq E$ with $|\lambda| = n$,
\[
\det(A_\lambda) \neq 0 \iff \lambda \in \mathcal{B}(M),
\]
where $A_\lambda$ denotes the corresponding $n \times n$ matrix obtained by using as columns some representatives $v_i' \in \mathbb{F}^n$ of the vectors $v_i$ for $i \in \lambda$.

Thus, the set of realizations is given by
\[
\widetilde{\mathcal{R}}(M;\mathbb{F}) \subset (\mathbb{P}^{n-1}_{\mathbb{F}})^d,
\]
a Zariski locally closed subset defined by determinantal equalities and non-vanishing conditions.
Two realizations differing by a projective change of coordinates define the same arrangement. This corresponds to the natural diagonal action of $\mathrm{PGL}(n,\mathbb{F})$ on $\PP=(\mathbb{P}^{n-1}_{\mathbb{F}})^d,$ which induces a natural action on
$\widetilde{\mathcal{R}}(M;\mathbb{F})$. The \emph{realization space} of $M$ is then defined as the quotient
\[
\mathcal{R}(M;\mathbb{F}) := \widetilde{\mathcal{R}}(M;\mathbb{F}) / \mathrm{PGL}(n,\mathbb{F}).
\]
The space $\mathcal{R}(M;\mathbb{F})$ is a moduli space parametrizing all realizations of $M$ up to projective equivalence. In particular, it can be viewed as a quasi-affine variety, that is an open subset of an affine variety, when the field $\mathbb{F}$ is algebraically closed of characteristic $0$.
 Indeed, one can consider only essential arrangements, and then the set $\mathcal{B}(M)$ is non-empty. Fixing $\lambda_0\in\mathcal{B}(M)$, we use $\mathrm{PGL}(n,\mathbb{F})$ to send $v_i$, $i\in\lambda_0$, to the standard basis, and then the torus to set further coordinates equal to $1$. This identifies $\mathcal{R}(M;\mathbb{F})$ with a locally closed subset of an affine space, see \cite{OSCAR}.
\begin{prop}
\label{ellp}
The realization space $\mathcal{R}(\mathcal{A},\mathbb{C})$ of arrangements described in Example \ref{ex1} is a Zariski-dense open subset in $\mathbb{A}^{5}_{\mathbb{C}}$.
\end{prop}
\begin{proof}
Using the \texttt{OSCAR} procedure presented in \cite{OSCAR}, we can compute the realization space $\mathcal{R}(\mathcal{A},\mathbb{C})$. The columns of the matrix below correspond to normal vectors of arrangements whose intersection lattice is isomorphic to that of $\mathcal{A}$. We have
\[
\left[
\begin{array}{ccccccccccc}
1 & 1 & 0 & 0 & 0 & 0 & 0 & 0 & 0 & 0 & 0 \\
0 & 1 & x_1 x_5 - x_4 x_5 & 1 & x_2 & x_5 & 1 & 1 & 0 & 0 & 0 \\
0 & 1 & x_1 x_3 x_5 - x_1 x_4 + x_1 - x_3 x_5 & 0 & x_3 & x_1 & x_3 & x_3 & 1 & 1 & 0 \\
0 & 0 & x_1 x_5 - x_4 x_5 & 0 & x_2 & x_1 x_5 & x_2 & x_4 & x_5 & 0 & 1
\end{array}
\right]
\]
in the multivariate polynomial ring
\[
\mathbb{C}[x_1, x_2, x_3, x_4, x_5],
\]
avoiding the zero loci of the following forbidden locus of polynomials:
\[
\begingroup
\setlength{\jot}{0pt}
\begin{aligned}
{\rm ForbiddenLocus}(\mathcal{R}(\mathcal{A},\mathbb{C}))= \{&x_1 - x_3 x_5,\; x_2 - 1,\; x_3,\; x_4 - 1,\; x_3 - 1,\; x_1 x_2 - x_3 x_5,\\
&x_1,\; x_1 - 1,\; x_1 - x_5,\\
&x_1 x_2 x_3 x_5 - x_1 x_2 x_4 + x_1 x_2 - x_1 x_3 x_5 - x_2 x_3 x_5 + x_3 x_4 x_5,\\
&x_2 - x_3,\; x_1 x_3 x_5 - x_1 x_4 + x_1 - x_3 x_5,\\
&x_1 x_3 x_5 - x_1 x_4 - x_1 x_5 + x_1 - x_3 x_5 + x_4 x_5,\\
&x_5,\; x_2 - x_4,\; x_1 - x_4,\; x_4,\; x_1 - x_2,\; x_2,\\
&x_3 x_5 - x_4,\; x_2 x_3 x_5 - x_2 x_4 + x_2 - x_3 x_5,\\
&x_3 x_5 - x_4 - x_5,\; x_2 - x_3 x_5,\\
&x_1 x_2 - x_1 x_4 - x_2 x_4 + x_3 x_4 x_5 - x_3 x_5 + x_4,\\
&x_2 - x_3 x_5 + x_5,\; x_2 x_5 + x_2 - x_3 x_5,\\
&x_1 x_3 x_5 - x_1 x_4 - x_1 x_5 - x_3 x_5 + x_4 x_5 + x_4,\\
&x_1 x_3 x_5 - x_1 x_4 - x_1 x_5 + x_4 x_5,\\
&x_2 x_3 + x_2 x_4 - x_2 - x_3 x_4,\\
&x_1 x_2 - x_1 x_3 x_5 + x_1 x_5 - x_2 x_5,\\
&x_1 x_2 x_3 x_5 - x_1 x_2 x_4 - x_1 x_2 x_5 + x_1 x_2 - x_1 x_3 x_5 + x_1 x_5\\
&\quad - x_2 x_3 x_5 + x_2 x_4 x_5 + x_3 x_4 x_5 - x_4 x_5,\\
&x_1 x_2 x_5 + x_1 x_2 - x_1 x_3 x_5 - x_2 x_5\}.
\end{aligned}
\endgroup
\]
\end{proof}
\noindent
Let us work in the setting of Proposition~\ref{ellp}. Denote by
\[
Q(x_{1},x_{2},x_{3},x_{4},x_{5})
\]
the defining equation of the arrangement
\[
\mathcal{A}_{(x_{1},x_{2},x_{3},x_{4},x_{5})},
\]
viewed as a point in the realization space $\mathcal{R}(\mathcal{A},\mathbb{C})$.

For a sufficiently general choice of parameters \((x_{1},x_{2},x_{3},x_{4},x_{5})\), the graded Betti numbers of the module \(M(Q(x_{1},x_{2},x_{3},x_{4},x_{5}))\) are expected to be as stated below. A direct computation at the point \((2,3,5,7,11)\) gives the graded Betti numbers
\[
\mathbf{d}=(2,7_{13}), \qquad 
\mathbf{c}=(8_{16}), \qquad 
\mathbf{b}=(9_{5}).
\]

It is natural to consider the locus \(\mathcal{X} \subset \mathcal{R}(\mathcal{A}, \mathbb{C})\) consisting of those admissible quintuples \((x_{1},x_{2},x_{3},x_{4},x_{5})\) for which the graded Betti numbers of \(M(Q(x_{1},x_{2},x_{3},x_{4},x_{5}))\) take the form
\[
\mathbf{d}=(2,6_{2},7_{7}), \qquad 
\mathbf{c}=(7,8_{7},9_{2}), \qquad 
\mathbf{b}=(9_{2},10).
\]

We now describe this locus. Throughout, we read the four rows of the matrix in Proposition~\ref{ellp} as the coefficients of $w,x,y,z$, respectively. Then the first two columns give the planes $w=0$ and $w+L=0$, where $L=x+y$, while the remaining nine columns have vanishing $w$-coefficient. Hence, for every point $\xi=(x_1,\ldots,x_5)\in\mathcal{R}(\mathcal{A},\mathbb{C})$, we have
\begin{equation}
\label{eqfwLg}
Q(\xi)=f=w(w+L)\,g, \qquad g=g_\xi\in R=\mathbb{C}[x,y,z], \ \deg g=9,
\end{equation}
that is, nine planes of $\mathcal{A}_\xi$ pass through the point $O=[1:0:0:0]$ (in the coordinates $(w:x:y:z)$) and form the cone over a line arrangement $\mathcal{B}_\xi:g=0$ in $\mathbb{P}^2$, exactly as in Example~\ref{ex1}. By construction, $\mathcal{B}_\xi$ has the same combinatorics as the Ziegler arrangements $\mathcal{B}$ and $\mathcal{B}'$ in Example~\ref{exI1}, i.e. it consists of the six sides and the three main diagonals of a hexagon.

The following lemma explains how syzygies of $g$ produce syzygies of $f$. We denote by $E_0=x\partial_x+y\partial_y+z\partial_z$ the Euler derivation of $R$, and we regard derivations of $R$ as derivations of $S$ killing $w$.

\begin{lem}
\label{lemLift}
Let $g\in R$ be a reduced homogeneous polynomial of degree $e\geq 1$, let $L\in R_1$ be a non-zero linear form and set $f=w(w+L)g\in S$. Then
\[
\rho=(2w+L)E_0-w\big((e+1)L+ew\big)\partial_w\in D_0(f)_2 .
\]
Moreover, for every $0\neq \theta\in D_0(g)_r$, the derivations
\[
\delta_1=(2w+L)\,\theta-w\,\theta(L)\,\partial_w, \qquad
\delta_2=(e+2)L\,\theta-2\,\theta(L)\,E_0+e\,w\,\theta(L)\,\partial_w
\]
belong to $D_0(f)_{r+1}$, and
\[
\dim D_0(f)_{r+1}\ \geq\ \dim\big(S_{r-1}\rho\oplus\mathbb{C}\delta_1\oplus\mathbb{C}\delta_2\big)=\binom{r+2}{3}+2 .
\]
\end{lem}

\begin{proof}
Since $\theta(w)=0$, $\theta(g)=0$ and $E_0(g)=eg$, we get
\[
\theta(f)=w\,g\,\theta(L),\qquad E_0(f)=w\,g\,\big((e+1)L+ew\big),\qquad \partial_w f=g\,(2w+L).
\]
Hence
\[
\rho(f)=g\Big[(2w+L)\,w\big((e+1)L+ew\big)-w\big((e+1)L+ew\big)(2w+L)\Big]=0,
\]
\[
\delta_1(f)=g\Big[(2w+L)\,w\,\theta(L)-w\,\theta(L)(2w+L)\Big]=0,
\]
\[
\delta_2(f)=w\,g\,\theta(L)\Big[(e+2)L-2(e+1)L-2ew+2ew+eL\Big]=0.
\]
Clearly $\rho\neq 0$, hence $\dim S_{r-1}\rho=\dim S_{r-1}=\binom{r+2}{3}$. Assume now that
\[
a\,\delta_1+b\,\delta_2=P\rho
\]
with $a,b\in\mathbb{C}$ and $P\in S_{r-1}$. We compare the components along $\partial_x,\partial_y,\partial_z$, which we view as triples of polynomials, and we write $E_0=(x,y,z)$ for short. This gives
\[
a(2w+L)\,\theta+b\big((e+2)L\,\theta-2\,\theta(L)\,E_0\big)=P\,(2w+L)\,E_0 .
\]
Applying the $R$-algebra morphism $S\to R$, $w\mapsto -L/2$, we get $b\big((e+2)L\theta-2\theta(L)E_0\big)=0$. If $b\neq 0$, we apply this derivation of $R$ to $g$ and obtain $-2e\,\theta(L)\,g=0$, hence $\theta(L)=0$, hence $L\theta=0$, i.e. $\theta=0$, a contradiction. Therefore $b=0$, and then $a\,\theta=P\,E_0$ after cancelling the non-zero factor $2w+L$. Applying both sides to $g$ we get $0=e\,P\,g$, hence $P=0$ and finally $a=0$.
\end{proof}

Now we can describe the relevant part of $\mathcal{R}(\mathcal{A},\mathbb{C})$ explicitly. Consider the hypersurface
\[
\mathcal{H}:\ x_1x_2-x_1-x_2+x_4=0, \quad \text{ i.e. } \quad x_4=1-(1-x_1)(1-x_2),
\]
in $\mathbb{A}^5_{\mathbb{C}}$. Note that the polynomial $x_1x_2-x_1-x_2+x_4$ does not occur in ${\rm ForbiddenLocus}(\mathcal{R}(\mathcal{A},\mathbb{C}))$, and that no element of ${\rm ForbiddenLocus}(\mathcal{R}(\mathcal{A},\mathbb{C}))$ vanishes identically on $\mathcal{H}$. Since $\mathcal{H}\cong\mathbb{A}^4_{\mathbb{C}}$ via the coordinates $(x_1,x_2,x_3,x_5)$, it follows that $\mathcal{H}\cap\mathcal{R}(\mathcal{A},\mathbb{C})$ is a non-empty Zariski open subset of $\mathcal{H}$ and, in particular, it is irreducible of dimension $4$.

\begin{prop}
\label{propX}
With the above notation, the following hold.
\begin{enumerate}
\item For $\xi\in\mathcal{R}(\mathcal{A},\mathbb{C})$, the six triple points of $\mathcal{B}_\xi$ lie on a conic if and only if $\xi\in\mathcal{H}$.
\item For any $\xi\in\mathcal{H}\cap\mathcal{R}(\mathcal{A},\mathbb{C})$ one has $\dim D_0(Q(\xi))_6\geq 37$.
\item There is a non-empty Zariski open subset $U\subset\mathcal{R}(\mathcal{A},\mathbb{C})$ such that $\dim D_0(Q(\xi))_6=35$ for all $\xi\in U$.
\end{enumerate}
In particular, the locus $\mathcal{X}_{ng}\subset\mathcal{R}(\mathcal{A},\mathbb{C})$ of arrangements whose graded Betti numbers are not the generic ones contains $\mathcal{H}\cap\mathcal{R}(\mathcal{A},\mathbb{C})$ and is contained in a proper Zariski closed subset of $\mathcal{R}(\mathcal{A},\mathbb{C})$, hence
\[
\dim\mathcal{X}_{ng}=4 .
\]
\end{prop}

\begin{proof}
(1) We represent the triple points of $\mathcal{B}_\xi$ as cross products of the normal vectors of two of the lines through them, namely for the pairs of columns $(3,5)$, $(4,5)$, $(9,10)$, $(3,6)$, $(4,6)$, $(7,8)$, and we let $\Delta(\xi)$ be the determinant of the $6\times 6$ matrix whose rows are the values of the monomials $x^2,y^2,z^2,xy,xz,yz$ at these six points. Then the six triple points lie on a conic if and only if $\Delta(\xi)=0$, and a direct computation gives
\[
\begin{aligned}
\Delta=\;&-x_1^2x_3x_5^5(x_1-1)^2(x_1-x_3x_5)(x_2-x_4)^2(x_2-x_3x_5)(x_3x_5-x_4)\\
&\times(x_1x_2-x_1-x_2+x_4)\,(x_1x_2x_3x_5-x_1x_2x_4+x_1x_2-x_1x_3x_5-x_2x_3x_5+x_3x_4x_5)^2 .
\end{aligned}
\]
All the factors except $x_1x_2-x_1-x_2+x_4$ belong to ${\rm ForbiddenLocus}(\mathcal{R}(\mathcal{A},\mathbb{C}))$, hence they do not vanish on $\mathcal{R}(\mathcal{A},\mathbb{C})$, which proves the claim.

(2) Let $\xi\in\mathcal{H}\cap\mathcal{R}(\mathcal{A},\mathbb{C})$. By (1) and the discussion above, $\mathcal{B}_\xi$ is a Ziegler arrangement as in Example~\ref{exI1} whose six vertices lie on a conic, hence ${\rm mdr}(g_\xi)=5$ by \cite[Corollary 1.7]{ZP}.
Using \eqref{eqfwLg} and Lemma~\ref{lemLift} with $e=9$, $L=x+y$ and $0\neq\theta\in D_0(g_\xi)_5$, we get
\[
\dim D_0(Q(\xi))_6\geq\binom{7}{3}+2=37 .
\]

(3) The function $\xi\mapsto\dim D_0(Q(\xi))_6$ is upper semicontinuous on $\mathcal{R}(\mathcal{A},\mathbb{C})$, since it equals $4\dim S_6-\rank\big(\phi_\xi\big)$, where $\phi_\xi:S_6^4\to S_{16}$ is the linear map $(a,b,c,d)\mapsto a\,\partial_wQ+b\,\partial_xQ+c\,\partial_yQ+d\,\partial_zQ$, whose matrix entries are polynomials in $\xi$. By Lemma~\ref{lemLift}, we always have $\rho\in D_0(Q(\xi))_2$, hence $S_4\rho\subset D_0(Q(\xi))_6$, hence $\dim D_0(Q(\xi))_6\geq 35$. At the integral point $\xi_0=(2,3,5,7,11)$ the matrix of $\phi_{\xi_0}$, reduced modulo the prime $p=1000003$, has rank $301=4\cdot 84-35$. Since the rank of an integral matrix can only drop under reduction modulo $p$, we get $\dim D_0(Q(\xi_0))_6=35$, and hence the same equality holds on a non-empty Zariski open subset $U$ containing $\xi_0$. As $\mathcal{R}(\mathcal{A},\mathbb{C})$ is irreducible, $U$ is dense.

Finally, let $V\subset\mathcal{R}(\mathcal{A},\mathbb{C})$ be a non-empty Zariski open subset on which the graded Betti numbers of $M(Q(\xi))$ are constant (by generic flatness there is a non-empty open subset over which the family is flat, and graded Betti numbers are upper semicontinuous in flat families). Since $\mathcal{R}(\mathcal{A},\mathbb{C})$ is irreducible, $V\cap U\neq\emptyset$, so the generic value of $\dim D_0(Q(\xi))_6$ is $35$. Hence (2) gives $\mathcal{H}\cap\mathcal{R}(\mathcal{A},\mathbb{C})\subset\mathcal{X}_{ng}\subset\mathcal{R}(\mathcal{A},\mathbb{C})\setminus V$, and the claim on dimensions follows, since $\mathcal{H}\cap\mathcal{R}(\mathcal{A},\mathbb{C})$ has dimension $4$ and $\mathcal{R}(\mathcal{A},\mathbb{C})\setminus V$ has dimension at most $4$.
\end{proof}

\begin{rk}
\label{rkXcomp}
Proposition~\ref{propX} explains the Ziegler pair in Example~\ref{ex1}. Indeed, both arrangements are of the form \eqref{eqfwLg}, namely $f=w(w+L)g$ and $f'=w(w+L)g'$ with $L=x-3y+5z$, where
\[
g=xy(x-y-z)(x-y+z)(2x+y-2z)(x+3y-3z)(3x+2y+3z)(x+5y+5z)(7x-4y-z)
\]
and
\[
g'=xy(4x-5y-5z)(x-y+z)(16x+13y-20z)(x+3y-3z)(3x+2y+3z)(x+5y+5z)(7x-4y-z).
\]
A direct computation shows that the six triple points of the line arrangement $g=0$ lie on a conic, while those of $g'=0$ do not. Hence ${\rm mdr}(g)=5$ and ${\rm mdr}(g')=6$ by \cite[Corollary 1.7 and Theorem 1.8]{ZP}, and Lemma~\ref{lemLift} yields two derivations of degree $6$ in $D_0(f)$, which are not in $S_4\rho$. These account for the entry $6_2$ in the graded Betti numbers of $M(f)$, which have no analogue for $M(f')$. In terms of Proposition~\ref{propX}, the arrangement $\A$ corresponds to a point of $\mathcal{H}$, while $\A'$ corresponds to a point outside $\mathcal{H}$.

Moreover, the graded Betti numbers are constant on a dense Zariski open subset of the irreducible variety $\mathcal{H}\cap\mathcal{R}(\mathcal{A},\mathbb{C})$. Computations modulo a large prime at many random integral points $\xi\in\mathcal{H}\cap\mathcal{R}(\mathcal{A},\mathbb{C})$ show that $\dim D_0(Q(\xi))_k$, for $1\leq k\leq 10$, coincides with the values determined by the graded Betti numbers of $M(f)$. At random points $\xi\notin\mathcal{H}$, these dimensions coincide instead with the values determined by the graded Betti numbers of $M(f')$. This strongly suggests that the generic graded Betti numbers along $\mathcal{H}$ are those of $M(f)$, i.e. that $\mathcal{X}$ contains a dense open subset of $\mathcal{H}\cap\mathcal{R}(\mathcal{A},\mathbb{C})$. In that case $\dim\mathcal{X}=4$.
\end{rk}

\section{The Elliptic Matroids $T_n$}

Let $n \ge 4$. We define a rank-$3$ matroid $T_n$ on the ground set
\[
E := \mathbb{Z}/n\mathbb{Z}
\]
by declaring a triple $\{i,j,k\} \subset E$ to be dependent if and only if
\[
i + j + k \equiv 0 \pmod n.
\]
All other triples are bases.

\medskip

\noindent
This defines a matroid. Non-emptiness and uniform rank are immediate. For the exchange axiom, let $A=\{a,x,y\}$ and $B$ be bases with $a \in A \setminus B$. A set $\{x,y,b\}$ fails to be a basis if and only if $x+y+b \equiv 0$, which determines at most one $b \in E$. If $B\setminus A=\{b\}$, then $(A\setminus\{a\})\cup\{b\}=B$; otherwise $|B\setminus A|\geq 2$, so one can choose $b \in B \setminus A$ avoiding this value, and $(A\setminus\{a\}) \cup \{b\}$ is a basis.

\medskip

\noindent
The construction is invariant under affine transformations of the form
\[
i \longmapsto ui + a \,\text{ such that }\, u \in (\mathbb{Z}/n\mathbb{Z})^{*} \text{ and } 3a\equiv 0({\rm mod}\, n),
\]
so $\mathrm{Aut}(T_n)$ contains $(\mathbb{Z}/n\mathbb{Z})^{*}$.

\medskip

\noindent
A realization of $T_n$ over $\mathbb{C}$ is a family of points
\[
(p_i)_{i \in E} \subset \mathbb{P}^2
\]
such that
\[
p_i,p_j,p_k \text{ are collinear } \iff i+j+k \equiv 0 \pmod n.
\]
Equivalently, writing $p_i$ as homogeneous column vectors, the realization space is cut out by the conditions
\[
\det(p_i,p_j,p_k)=0 \quad \text{for all } i+j+k \equiv 0 \pmod n,
\]
together with the non-vanishing of the remaining $3\times 3$ minors.

\medskip

\noindent
Assume now that \(n\geq 7\). Then there exist four elements of \(E\)
which do not contain a dependent triple; for instance, one may take
\(\{0,1,2,3\}\). After a projective change of coordinates, every realization
in the corresponding open chart may therefore be written in the normalized
form
\[
p_0=[1:0:0],\qquad
p_1=[0:1:0],\qquad
p_2=[0:0:1],\qquad
p_3=[1:1:1].
\]
The remaining points are not uniquely determined by this normalization.
Rather, the incidence relations
\[
\det(p_i,p_j,p_k)=0
        \qquad\text{for } i+j+k\equiv 0 \pmod n
\]
express the remaining points in terms of one auxiliary parameter, together
with the non-vanishing conditions excluding coincidences and additional
collinearities.

Geometrically, these incidence relations force the points \(p_i\) to lie on
a plane cubic curve. In the non-degenerate case, after choosing a flex as origin on
a smooth cubic \(E\), the collinearity condition is equivalent to the group-law
condition
\[
p_i+p_j+p_k=0
        \qquad\text{on } E.
\]
Thus non-degenerate realizations of \(T_n\) are obtained from pairs
\((E,P)\), where \(E\) is a smooth plane cubic and \(P\in E\) is a point of
exact order \(n\), by setting
\[
p_i=iP,\qquad i\in \mathbb Z/n\mathbb Z.
\]
One has to remove the finite degeneracy locus where some of the points
coincide or where extra collinearities occur. Consequently, for \(n\geq 10\), after removing the
boundary and the degenerate configurations, the moduli space of non-degenerate
realizations of \(T_n\) is an open part of \(X_1(n)\) -- for details please consult \cite{BR}.

We now briefly discuss the geometry of realizations, i.e. line arrangements realizing the matroids $T_n$. These configurations are closely related to classical constructions arising from the orchard problem. In particular, they can be traced back to the work of Burr, Gr\"unbaum, and Sloane \cite{Burr}, where geometric properties of plane cubic curves were used to construct extremal configurations of points and lines with many triple intersections.

In the case of $T_n$, the defining condition
\[
i + j + k \equiv 0 \pmod n
\]
has a direct geometric interpretation: the corresponding lines $\ell_i, \ell_j, \ell_k$ intersect in a common point. Thus, this condition encodes precisely the triple intersection structure of the arrangement.

From a numerical perspective, this yields an arrangement with exactly
\[
n_3 = 1 + \left\lfloor \frac{n(n-3)}{6} \right\rfloor
\]
triple points \cite[A058212]{OEIS}, while all remaining intersection points are double points.

Having established this description, we now turn to the specific case of $n=10$, which was previously studied in \cite{CP}.
\vskip0.3cm
\paragraph{\bf{Matroid} $T_{10}$.}

Consider the family of line arrangements defined by
\[
\begin{aligned}
Q_{t}(x,y,z)=\;& xyz(x+y+z)(y+z)
((t-2)x + (t^{2}-t-1)y + t(t-2)z)\\
&\times ((t-2)(x+y) + t(t-2)z)
(t(t-2)x + (t^{2}-t-1)y + t(t-2)z)\\
&\times ((t-2)x + (t-1)(t-2)z)
((t-2)(1-t)x - (t-1)^{2}y),
\end{aligned}
\]
subject to the condition $t \not\in\bigg\{0,1,2, (1\pm \sqrt{5})/2 \bigg\}$. 

For any such $t$, the arrangement $Q_t=0$ realizes $T_{10}$, and every realization of $T_{10}$ is projectively equivalent to exactly one of them. Indeed, no three of the lines $x,y,z,x+y+z$ are concurrent, so after a projective change of coordinates any realization contains these four lines, with the same labels. The incidence conditions then force the fifth line to be $y+z=0$. Normalizing the sixth line so that its coefficient of $x$ is $1$, and calling $t$ its coefficient of $z$, the remaining incidence conditions determine the last five lines uniquely in terms of $t$, and they are exactly the lines of $Q_t$. Hence the realization space of the matroid $T_{10}$ is isomorphic to the complex line with finitely many points removed, corresponding to degenerate configurations (i.e., arrangements with distinct intersection lattices). Moreover, this family admits at least two homologically distinguished realizations, namely 
\[
L_{10}^1: Q_{3}=0, \qquad
L_{10}^2: Q_{\sqrt5+3}=0.
\]
The homological behavior of these arrangements is captured by the minimal graded free resolution of the module of Jacobian relations $D_0(Q_{t})$. One computes:
\[
0 \to R(-8) \to R(-6)^{\oplus 2} \oplus R(-5) \to D_0(Q_{3})
\]
for $L_{10}^1$, and
\[
0 \to R(-8)\oplus R(-7) \to R(-7)\oplus R(-6)^{\oplus 2}\oplus R(-5) \to D_0(Q_{\sqrt5+3})
\]
for $L_{10}^2$. In particular, $L_{10}^1$ is a curve of type $2A$, whereas $L_{10}^2$ is of type $2B$; see \cite{ADP} for a detailed discussion of curves of type $2$. Since they have identical values of ${\rm mdr}$, both arrangements satisfy the ${\rm (MDR)}$ condition. 
The pair $(L_{10}^1, L_{10}^{2})$ satisfies ${\rm (HF)}$ condition, namely we have
\begin{multline*}
HF(Q_{3}) = HF\bigg(Q_{\sqrt5+3} \bigg) = 1 + 3t + 6t^2 +10t^3 + 15t^4 + 21t^5 + 28t^6 +36t^7 \\ + 45t^8 + 52t^9 + 57t^{10} + 60t^{11} + 61t^{12} +60t^{13} + 58t^{14} + 57t^{15} + \ldots    
\end{multline*} 

We conclude our analysis of the case $n=10$ by discussing the special value $t=3$. 
First, we observe that for this parameter the arrangement may exhibit the generic type of the resolution, namely that the generic element in the moduli space is of type $2A$, as we obtain the same results for many values of $t$.

However, at $t=3$ the arrangement acquires an additional geometric feature, reminiscent 
of the phenomenon observed in the first Ziegler pair of line arrangements. 
More precisely, one can compute the coordinates of the twelve triple points, which are
\[
\begin{aligned}
&[0:0:1],\quad [0:1:0],\quad [0:1:-1],\quad [0:3:-5],\\
&[1:0:0],\quad [1:0:-1],\quad [3:0:-1],\quad [1:-1:0],\\
&[2:-1:-1],\quad [2:-1:1],\quad [2:1:-1],\quad [6:-3:-1].
\end{aligned}
\]
It turns out that there exists a unique irreducible conic
\[
C:\quad x^2 + xy + xz + 4yz = 0
\]
passing through the following six triple points:
\[
[0:0:1],\quad
[0:1:0],\quad
[1:0:-1],\quad
[1:-1:0],\quad
[2:-1:1],\quad
[2:1:-1].
\]
It is clear that the fact an irreducible conic passes through six points is a rather surprising geometric property.

Concluding this part, let us point out that similar properties appear in the case $n = 12$. A feature shared by these two examples appears to be that their moduli space is rational, that is, it can be identified with the complex line with a finite number of points removed. It should be noted, however, that this does not hold in general for elliptic matroids: for example, when $n = 11$ or $n = 13$, the corresponding moduli spaces no longer appear to be rational curves.
\section*{Acknowledgments}
We would like to thank Lukas K\"uhne for his help with symbolic computations.


\begin{thebibliography}{00}

\bibitem{Abe}
T. Abe, Plus-one generated and next to free arrangements of hyperplanes. \textit{Int. Math. Res. Not. IMRN} \textbf{2021}, no.~12, 9233--9261 (2021).

\bibitem{ADP}
T. Abe, A. Dimca, P. Pokora, A new hierarchy for complex plane curves. \textit{Canad. Math. Bull.} \textbf{69}, no.~3, 824--848 (2026).

\bibitem{BR}
L. Borisov, X. Roulleau, Modular curves $X_{1}(n)$ as moduli spaces of point arrangements and applications. Preprint, arXiv:2404.04364 (2024).

\bibitem{Burr}
S. A. Burr, B. Gr\"unbaum, N. J. A. Sloane, The orchard problem. \textit{Geom. Dedicata} \textbf{2}, 397--424 (1974).

\bibitem{OSCAR}
D. Corey, L. K\"uhne, B. Schr\"oter, Matroids. In: W. Decker, C. Eder, C. Fieker, M. Horn, M. Joswig (eds.), \textit{The Computer Algebra System OSCAR}. Algorithms and Computation in Mathematics, vol.~32, Springer, Cham (2025).

\bibitem{CP}
M. Cuntz, P. Pokora, Singular plane curves: freeness and combinatorics. \textit{Innov. Incidence Geom.} \textbf{22}, no.~1, 47--65 (2025).

\bibitem{DmaxS}
A. Dimca, Freeness versus maximal degree of the singular subscheme for surfaces in $\mathbb{P}^3$. \textit{Geom. Dedicata} \textbf{183}, 101--112 (2016).

\bibitem{DHA}
A. Dimca, \textit{Hyperplane Arrangements: An Introduction}. Universitext, Springer, Cham (2017).

\bibitem{FS}
A. Dimca, G. Sticlaru, Free and nearly free surfaces in $\mathbb{P}^3$. \textit{Asian J. Math.} \textbf{22}, 787--810 (2018).

\bibitem{3syz}
A. Dimca, G. Sticlaru, Plane curves with three syzygies, minimal Tjurina curves, and nearly cuspidal curves. \textit{Geom. Dedicata} \textbf{207}, 29--49 (2020).

\bibitem{ZP}
A. Dimca, G. Sticlaru, From Pascal's Theorem to the geometry of Ziegler's line arrangements. \textit{J. Algebraic Combin.} \textbf{60}, 991--1009 (2024).

\bibitem{HD}
A. Dimca, G. Sticlaru, Bourbaki modules and the module of Jacobian derivations of projective hypersurfaces. \textit{Collect. Math.}, to appear. Preprint, arXiv:2506.23950 (2025).

\bibitem{Betti2}
A. Dimca, G. Sticlaru, Graded Betti numbers of the Jacobian algebra and total Tjurina numbers of plane curves. Preprint, arXiv:2601.08583 (2026).

\bibitem{Betti3}
A. Dimca, G. Sticlaru, Graded Betti numbers of the Jacobian algebra of surfaces in $\mathbb{P}^3$. Preprint, arXiv:2602.09966 (2026).

\bibitem{Eis}
D. Eisenbud, \textit{The Geometry of Syzygies: A Second Course in Algebraic Geometry and Commutative Algebra}. Graduate Texts in Mathematics, vol.~229, Springer, New York (2005).

\bibitem{MNS}
J. Migliore, U. Nagel, H. Schenck, Schemes supported on the singular locus of a hyperplane arrangement in $\mathbb{P}^n$. \textit{Int. Math. Res. Not. IMRN} \textbf{2022}, no.~1, 18941--18971 (2022).

\bibitem{OEIS}
OEIS Foundation Inc., The On-Line Encyclopedia of Integer Sequences. \url{https://oeis.org} (2019).

\bibitem{OT}
P. Orlik, H. Terao, \textit{Arrangements of Hyperplanes}. Grundlehren der mathematischen Wissenschaften, vol.~300, Springer-Verlag, Berlin (1992).

\bibitem{R}
R. Randell, Milnor fibrations of lattice-isotopic arrangements. \textit{Proc. Amer. Math. Soc.} \textbf{125}, 3003--3009 (1997).

\end{thebibliography}
\end{document}